\documentclass[10pt]{article}

\usepackage[
	a4paper,
	margin=1in
]{geometry}
\usepackage{setspace}

\usepackage{sectsty} 
\usepackage{titlesec}
\allsectionsfont{\boldmath} 

\usepackage{appendix}

\usepackage{amsmath, amssymb, amsthm}
\usepackage{mathtools}

\usepackage[
	setpagesize=false,
	colorlinks=true,
	linkcolor=BrickRed,
        citecolor=OliveGreen,
        urlcolor=black,
	pdfencoding=auto,
	psdextra,
]{hyperref}
\usepackage{cleveref}

\usepackage{etoolbox} 
\usepackage[shortlabels]{enumitem}
\usepackage{xparse} 
\usepackage{blkarray,multirow}

\usepackage{graphicx}
\usepackage[dvipsnames]{xcolor}
\usepackage{tikz}
\usetikzlibrary{patterns}
\usetikzlibrary{decorations.markings, arrows.meta}

\usepackage{todonotes}
\usepackage[
    deletedmarkup=sout,
    commentmarkup=uwave,
    authormarkuptext=name
]{changes}

\usepackage{comment}

\usepackage{tocloft}

\setlist[enumerate,1]{label=(\arabic*), ref=(\arabic*)}
\setlist[enumerate,3]{label=(\roman*), ref=(\roman*)}

\theoremstyle{plain}
\newtheorem{theorem}{Theorem}[section]
\newtheorem{lemma}[theorem]{Lemma}
\newtheorem{corollary}[theorem]{Corollary}
\newtheorem{proposition}[theorem]{Proposition}

\newtheorem{conjecture}[theorem]{Conjecture}

\newtheorem*{claim*}{Claim}



\theoremstyle{definition}
\newtheorem{definition}[theorem]{Definition}

\newcommand{\calC}{\mathcal{C}}

\newcommand{\calE}{\mathcal{E}}
\newcommand{\calF}{\mathcal{F}}

\newcommand{\calH}{\mathcal{H}}
\newcommand{\calL}{\mathcal{L}}

\newcommand{\calP}{\mathcal{P}}

\newcommand{\calS}{\mathcal{S}}

\newcommand{\ve}{\varepsilon}

\newcommand{\ori}[1]{\overrightarrow{#1}}
\newcommand{\dir}[1]{\overleftrightarrow{#1}}

\newcommand{\defeq}{\coloneqq}

\title{A counting version of Petersen's $2$-factor theorem}

\author{
Hyunwoo Lee%
        \thanks{Department of Mathematical Sciences, KAIST, South Korea and Extremal Combinatorics and Probability Group (ECOPRO), Institute for Basic Science (IBS).
        E-mail: {\ttfamily hyunwoo.lee@kaist.ac.kr.} Supported by the National Research Foundation of Korea (NRF) grant funded by the Korea government(MSIT) No. RS-2023-00210430, and the Institute for Basic Science (IBS-R029-C4).}
}

\usepackage[square,sort,comma,numbers]{natbib}
\begin{document}
\maketitle

\begin{abstract}
    A classical result of Petersen states that every regular graph of even degree has a $2$-factor. We prove that every $n$-vertex $2r$-regular simple graph contains at least $\left((1 + o_r(1))\frac{2r}{e}\right)^n$ distinct $2$-factors. This improves the previously known lower bound $\left((1 + o_r(1))\frac{r}{e}\right)^n$ by a factor of $2^{(1 + o(1))n}$ and is asymptotically tight for large $r$. As a direct consequence, we determine asymptotically tight bounds on the number of $2$-factorizations of a given $2r$-regular simple graph for every sufficiently large $r$.
\end{abstract}


\section{Introduction}\label{sec:intro}

Petersen's $2$-factor theorem~\cite{Petersen}, a cornerstone of early graph theory, states that every $2r$-regular graph admits a decomposition into $r$ edge-disjoint $2$-factors. Equivalently, every regular graph of even degree contains a $2$-factor as a spanning subgraph. This naturally raises the quantitative question of how many distinct $2$-factors a regular graph of even degree must contain. Although Petersen's theorem also holds for multigraphs, we restrict our attention throughout to simple graphs to avoid unnecessary technical complications.

An influential precedent for this type of counting question is provided by another classical theorem of Petersen~\cite{Petersen}, which states that every bridgeless cubic graph has a perfect matching. Its celebrated quantitative strengthening, the Lov\'{a}sz--Plummer conjecture~\cite[Conjecture 8.1.8]{Lovasz-Plummer}, proved by Esperet, Kardo\v{s}, King, Kr\'{a}l, and Norine~\cite{Lovasz-Plummer-Proof}, asserts that every bridgeless cubic graph contains exponentially many perfect matchings. Taking complements within the edge set of a cubic graph gives a bijection between perfect matchings and $2$-factors. Thus, the Lov\'{a}sz--Plummer conjecture can equivalently be viewed as an exponential counting theorem for $2$-factors in bridgeless cubic graphs. For another example, Schrijver~\cite{Schrijver-1-factor} established an asymptotically tight lower bound on the number of $1$-factors in an $r$-regular bipartite graph, which is a counting version of a classical graph theory result. Also, each $1$-factorization of a complete bipartite graph corresponds to a Latin square; the number of Latin squares of order $n$ was calculated by estimating the number of $1$-factorizations of complete bipartite graphs~\cite{Schrijver-Latin,Cameron-Latin}.

Returning to $2$-factors in regular graphs of even degree, a standard argument based on the Egorychev--Falikman theorem~\cite{Egorychev,Falikman} shows that every $n$-vertex $2r$-regular graph contains at least
$\left((1 + o_r(1))\frac{r}{e}\right)^n$ distinct $2$-factors~\cite[Lemma~18]{Knox-Kuhn-Osthus};
see also~\cite[proof of Lemma~2]{Frieze-Krivelevich}. On the other hand, the Br\'{e}gman--Minc inequality~\cite{Minc,Bregman} gives the upper bound $\left((1+o_r(1))\frac{2r}{e}\right)^n$, leaving an exponential gap of roughly $2^n$ between the best-known general bounds. Robalewska~\cite{Robalewska} determined the asymptotic distribution of the number of $2$-factors in a random regular graph. In particular, her result implies that, at the exponential scale, a typical $2r$-regular graph contains $\left((1+o_r(1))\frac{2r}{e}\right)^n$ many $2$-factors. Thus, the general upper bound captures the typical behavior, whereas the corresponding worst-case lower bound remained open.

Our main result closes this gap and thereby establishes an asymptotically tight counting version of Petersen's $2$-factor theorem. To state our results, for every integer $s\geq 1$ and graph $G$, we write $\calF_s(G)$ for the set of $s$-factors of $G$. For later purposes, for a given digraph $D$, we also denote by $\calF_s(D)$ the set of spanning subdigraphs of $D$ where both the out- and in-degrees are $s$\footnote{We use the same notation, $s$-regular, for digraphs with this property.} for all the vertices.

\begin{theorem}\label{thm:main}
    Let $G$ be an $n$-vertex $2r$-regular graph. Then we have
    \begin{equation*}
        |\calF_2(G)| = \left((1 + o_r(1)) \frac{2r}{e} \right)^n.
    \end{equation*}
\end{theorem}

We remark that the term $o_r(1)$ is unavoidable. Indeed, for small values of $r$, the lower bounds and upper bounds on $|\calF_2(G)|$ could be exponential. For example, suppose that $n$ is divisible by $40$. Let $G_1$ be the disjoint union of $n/8$ copies of $K_{4,4}$, and let $G_2$ be the disjoint union of $n/5$ copies of $K_5$. Then $|\calF_2(G_1)| \approx 1.755^n$ while $|\calF_2(G_2)| \approx 1.644^n$, which differ by an exponential factor.

By repeatedly applying Theorem~\ref{thm:main}, we deduce the following corollary on the number of distinct $2$-factorizations of $2r$-regular graphs.

\begin{corollary}\label{cor:factorization}
    Let $G$ be an $n$-vertex $2r$-regular graph. Then the number of distinct (ordered) $2$-factorizations of $G$ is 
    \begin{equation*}
        \left((1 + o_r(1)) \frac{2r}{e^2} \right)^{rn}.
    \end{equation*}
    
\end{corollary}

We note that Corollary~\ref{cor:factorization} is also asymptotically tight for large $r$. In detail, Corollary~\ref{cor:factorization} implies that every $n$-vertex $2r$-regular graph has $\left( (1 + o_r(1))\frac{2r}{e^2} \right)^{rn}$ distinct $2$-factorizations, by Stirling's approximation.


\section{Permanent inequalities}\label{sec:permanents}

A \emph{permanent} of an $n$ by $n$ square matrix $A$ is defined by
$$
    \operatorname{per}(A) \defeq \sum_{\sigma \in S_n} \prod_{i \in [n]} A_{i,\sigma(i)},
$$
which is closely related to several counting problems in graph theory. To prove our main results, we need two celebrated inequalities on permanents, so-called the Br\'{e}gman--Minc inequality~\cite{Minc,Bregman} and the Egorychev--Falikman theorem~\cite{Egorychev,Falikman}. Resolving a conjecture of Minc~\cite{Minc}, Br\'{e}gman~\cite{Bregman} proved the following theorem on the upper bound of permanents of $0$--$1$ matrices.

\begin{theorem}[Br\'{e}gman--Minc inequality~\cite{Minc,Bregman}]\label{thm:Bregman-Minc}
    Let $A$ be an $n$ by $n$ square matrix whose entries are either $0$ or $1$. Assume each $i\in [n]$ has row sum $r_i$. Then we have
    \begin{equation*}
        \operatorname{per}(A) \leq \prod_{i\in [n]} (r_i!)^{\frac{1}{r_i}}.
    \end{equation*}
\end{theorem}
A useful equivalent statement of Theorem~\ref{thm:Bregman-Minc} is the following upper bound on the number of perfect matchings in the bipartite graph. Let $H$ be a bipartite graph on the bipartition $L \cup R$. Then the number of perfect matchings in $H$ is at most $\prod_{v\in L} (d_H(v)!)^{1/d_H(v)}$. This equivalence is directly verified by considering a bipartite adjacency matrix. 
Also, Theorem~\ref{thm:Bregman-Minc} is particularly useful to obtain an upper bound of $|\calF_2(G)|$, where $G$ is an $n$-vertex $2r$-regular graph since the permanent of adjacency matrix of $G$ at least $|\calF_2(G)|$. From Theorem~\ref{thm:Bregman-Minc}, we obtain the following proposition, which proves the upper bound of Theorem~\ref{thm:main}.

\begin{proposition}\label{prop:BM-upper}
    Let $G$ be an $n$-vertex $2r$-regular graph and $A_G$ be the adjacency matrix of $G$. Then we have
    \begin{equation*}
        |\calF_2(G)| \leq \operatorname{per}(A_G) \leq \left( (2r)! \right)^{\frac{n}{2r}}.
    \end{equation*}
\end{proposition}

We say a square matrix $A$ is \emph{doubly stochastic} if all elements of $A$ are nonnegative and all the row sums and column sums are equal to $1$. As a counterpart to Theorem~\ref{thm:Bregman-Minc}, Egorychev~\cite{Egorychev} and Falikman~\cite{Falikman} independently proved the following lower bound for the permanent of a doubly stochastic matrix, which was posed by van der Waerden~\cite{vanderWaerden}.

\begin{theorem}[Egorychev--Falikman theorem~\cite{Egorychev,Falikman}]\label{thm:egorychev-falikman}
    Let $A$ be an $n$ by $n$ doubly stochastic matrix. Then we have
    \begin{equation*}
        \operatorname{per}(A) \geq \frac{n!}{n^n} \geq e^{-n}.
    \end{equation*}
\end{theorem}
One can view Theorem~\ref{thm:egorychev-falikman} as a weighted count version of the folklore fact that every regular bipartite graph has a $1$-factor. See also~\cite{Schrijver-1-factor}.
We note that the weaker bound, $\operatorname{per}(A) \geq e^{-n}$, was proved by Friedland~\cite{Friedland} before Egorychev and Falikman. For convenience in the calculation, we frequently use this slightly weaker inequality.

Both Theorems~\ref{thm:Bregman-Minc} and~\ref{thm:egorychev-falikman} play an important role in enumerative combinatorics since many combinatorial objects can be reformulated in terms of permanents. One such example is Schrijver's theorem~\cite{Schrijver-Euler} on the number of Eulerian orientations of a given graph $G$, which is one of the key ingredients of the proof of Theorem~\ref{thm:main}.
An \emph{Eulerian orientation} of a graph $G$ is an orientation of the edges of $G$ such that each vertex has the same number of outgoing edges and incoming edges. Denote by $\calE(G)$ the set of oriented graphs whose underlying graph is $G$ and whose orientation induces an Eulerian orientation of $G$.
With an elegant argument, Schrijver established the following bounds. 

\begin{theorem}[Schrijver~\cite{Schrijver-Euler}]\label{thm:Schrijver-EO}
    Let $G$ be an $n$-vertex $2r$-regular graph. Then we have
    \begin{equation*}
        \left( 2^{-r} \binom{2r}{r} \right)^n \leq |\calE(G)| \leq \binom{2r}{r}^{\frac{n}{2}}.
    \end{equation*}
\end{theorem}
The upper bound follows from the Brégman–Minc inequality, while an inductive vertex-splitting argument proves the lower bound.
For more results on the number of Eulerian orientations, we refer the readers to~\cite{Csikvari-Schweitzer,Borbenyi-Csikvari,BBMC,Isaev,LasVergnas,Isaev-McKay-Zhang}.


\section{Eulerian orientations after adding $2$-factors}

The key ingredient in our proof of Theorem~\ref{thm:main} is a lower bound on the ratio of the numbers of Eulerian orientations of two nested regular graphs. The following theorem shows that adding an edge-disjoint $2$-factor to an $n$-vertex $2(r-1)$-regular graph increases the number of Eulerian orientations by a factor of at least $(2-o_r(1))^n$, which may be of independent interest.

\begin{theorem}\label{thm:ratio-EO}
    For every real number $\ve > 0$, there exists $r_0 > 1$, depending only on $\ve$, such that the following holds for all $r \geq r_0$. Let $G$ and $G'$ be $n$-vertex $2(r+1)$- and $2r$-regular graphs, respectively, where $G'$ is a subgraph of $G$. Then we have
    $$
        \frac{|\calE(G)|}{|\calE(G')|} \geq (2 - \ve)^n.
    $$
\end{theorem}
Observe that a direct application of Theorem~\ref{thm:Schrijver-EO} does not suffice to prove Theorem~\ref{thm:ratio-EO} since the quantity $\left( 2^{-(r+1)} \binom{2(r+1)}{r+1}\right)^n / \binom{2r}{r}^{n/2}$ tends to $0$ as $n$ tends to infinity. The reason why naive application of Theorem~\ref{thm:Schrijver-EO} cannot work is that the structures of the extremal examples of the upper and lower bounds in Theorem~\ref{thm:Schrijver-EO} are significantly different. However, $G'$ is a subgraph of $G$ and also $G \setminus G'$ is a disjoint union of cycles, which is a very sparse graph. In this intuition, we use this structural information to prove Theorem~\ref{thm:ratio-EO}, together with Theorem~\ref{thm:Schrijver-EO} and the permanent inequalities discussed in Section~\ref{sec:permanents}.

In more detail, adding a spanning $2$-factor introduces $n$ new edges, so one might expect the number of orientations to increase by roughly $2^n$. The difficulty is that every vertex must remain balanced, which may require adjusting the directions of some original edges as well. We allow these adjustments inside a carefully chosen spanning subgraph of degree $2k$, where $k$ is large but fixed while the ambient degree $2r$ tends to infinity. The large ambient degree gives us enough freedom to choose many such subgraphs with few short cycles. Their simpler local structure allows us to estimate the cost of maintaining balance and recover almost a factor of two per vertex. Summing these comparisons over many choices of the subgraph, while controlling how often each orientation is counted, then transfers this gain to the original graphs. Thus, we first choose $k$ large enough to make the loss per vertex small, and then take $r$ sufficiently large to supply the required auxiliary subgraphs.

From now on, when we claim a statement holds if $0 < \alpha \ll \beta_1, \beta_2 \ll \gamma < 1 $, it means that there exist functions $f$ and $g$ such that the statement holds whenever $\beta_1, \beta_2 < f(\gamma)$ and $\alpha < g(\beta_1, \beta_2)$. We will not compute those functions explicitly. Also, we fix a constant $A_{r, k} \defeq \frac{(kr/e)^k}{(k!)^2}$.


\subsection{Counting oriented graphs via permanents}\label{subsec:count-ori-graphs}

Before proving Theorem~\ref{thm:ratio-EO}, we collect several useful lemmas that we will use later.
We start with the following sharp lower bound on the number of $k$-factors in a given $r$-regular oriented graph.

\begin{lemma}\label{lem:k-factor-lower}
    Fix parameters as
    $$
        0 < \frac{1}{r} \ll \frac{1}{k} \ll \ve < 1.
    $$
    Let $D$ be an $n$-vertex $r$-regular oriented graph. Let $r_0=k\lfloor r/k\rfloor$ and $s=r_0/k$, and write $(a)_k=a(a-1)\cdots(a-k+1)$. Then
    \begin{equation}\label{eq:factor-lower-exact}
        |\calF_k(D)|\geq \left(\frac{(r_0)_k^2}{(k!)^2(e s)^k}\right)^n.
    \end{equation}
    In particular,
    \begin{equation*}\label{eq:factor-lower-explicit}
        |\calF_k(D)| \geq \left((1-\ve)A_{r,k}\right)^n.
    \end{equation*}
\end{lemma}

We note that by iteratively applying the Br\'{e}gman--Minc inequality and the Egorychev--Falikman theorem, one could obtain a weaker lower bound, say $(1 - \ve)^{kn} A_{r, k}^n$. However, the latter argument requires a subexponential error bound that does not depend on $k$, as in the statement of Lemma~\ref{lem:k-factor-lower}. To this end, we partition the edge set of $D$ cleverly and avoid using the Br\'{e}gman--Minc inequality in the proof of Lemma~\ref{lem:k-factor-lower}, which would make the error term larger than our purpose.

We remark that Lemma~\ref{lem:k-factor-lower} can be obtained from a result of Csikv\'{a}ri and Schweitzer~\cite[Theorem 1.4]{Csikvari-Schweitzer}, but for the completeness of the paper, we include the proof here.

\begin{proof}[Proof of Lemma~\ref{lem:k-factor-lower}]
    First suppose that $k$ divides $r$, and write $s=r/k$. For every vertex $v$, partition its $r$ outgoing edges uniformly into $k$ labelled groups, each of size $s$. Independently partition its $r$ incoming edges in the same way. We make all $2n$ partitions independently, and denote their collection by $\pi$.

    Consider an auxiliary bipartite graph $B_\pi$ with left vertices $(v,i)_L$ and right vertices $(v,i)_R$, for $v\in V(D)$ and $1\leq i\leq k$. A directed edge $\ori{vw}$ becomes the edge from $(v,i)_L$ to $(w,j)_R$ if it belongs to outgoing group $i$ at $v$ and incoming group $j$ at $w$. We note that each directed edge of $D$ becomes exactly one edge and the auxiliary bipartite graph $B_\pi$ is $s$-regular, with $kn$ vertices in each class. Thus, its adjacency matrix $M_\pi$ satisfies
    \begin{equation}\label{eq:Mpi-EF}
        \operatorname{per}(M_\pi)\ge (s/e)^{kn}
    \end{equation}
    by the Egorychev--Falikman Theorem (Theorem~\ref{thm:egorychev-falikman}) since $\operatorname{per}(M_{\pi})$ is the same with the number of perfect matchings of $B_{\pi}$.
    Observe that a perfect matching of $B_\pi$ projects to a $k$-factor of $D$. In the opposite direction, a fixed $k$-factor occurs in this way precisely when its selected edges occupy all $k$ groups once at each end; in that case, it gives exactly one perfect matching. 
    For any one of the $2n$ partitions, the probability of this event for a prescribed vertex $v$ is
    $$
        \rho =\frac{k!(r-k)!/((s-1)!)^k}{r!/(s!)^k} =\frac{k!s^k}{(r)_k}.
    $$
    We computed $\rho$ by assigning the $k$ selected arcs bijectively to the labelled groups, then distributing the remaining arcs with $s-1$ in each group.
    Then independence and linearity of expectation now give the exact identity
    $$
        \mathbb{E}_\pi[\operatorname{per}(M_\pi)]=|\calF_k(D)|\rho^{2n}.
    $$
    Consequently, the above identity and \eqref{eq:Mpi-EF} imply
    \[
        |\calF_k(D)|\ge (s/e)^{kn}\rho^{-2n} =\left(\frac{(r)_k^2}{(k!)^2(e s)^k}\right)^n.
    \]

    For general $r$, we note that every $r$-regular oriented graph has a $1$-factorization, which is a classical result in graph theory. Then we arbitrarily choose $r_0$ such $1$-factors of $D$ and apply the divisible case to them. Then we obtain the desired inequality \eqref{eq:factor-lower-exact}.
    Finally, the lower bound $\left((1 - \ve)A_{r, k}\right)^n$ is followed by a simple numerical calculation under the condition $0< 1/r \ll 1/k \ll \ve$, so we omit the detail.
    This completes the proof.
\end{proof}

By using Lemma~\ref{lem:k-factor-lower}, we prove the following lemma.

\begin{lemma}\label{lem:edge-contain-prob}
    Fix parameters as
    $$
        0 < \frac{1}{r} \ll \frac{1}{k} \ll \ve < 1.
    $$
    Let $D$ be an $n$-vertex $r$-regular oriented graph. Let $S$ be a subdigraph of $D$ and let $F$ be a $k$-factor of $D$ that is chosen from $\calF_k(D)$ uniformly at random. Then we have
    \begin{equation*}
        \mathbb{P}(E(S)\subseteq E(F)) \leq (1 + \ve)^n \left( \frac{ek}{r} \right)^{|E(S)|}.
    \end{equation*}
\end{lemma}

To prove Lemma~\ref{lem:edge-contain-prob}, we estimate the upper bound on the number of $k$-factors in $D$ that contain all edges of $S$ by using the Br\'{e}gman--Minc inequality. For this, we define the following auxiliary bipartite graph.

\begin{definition}
    Let $D$ be an oriented graph and let $\mathbf{S} = \{(a_v, b_v)\}_{v\in V(D)}$ be a collection of pairs of non-negative integers $a_v$ and $b_v$ such that $\sum_{v\in V(D)} a_v = \sum_{v\in V(D)} b_v$. 
    Define an auxiliary bipartite graph $D_{\mathbf{S}}$ on a bipartition $L \cup R$ such that $L$ consists of $a_v$ copies of $v$ and $R$ consists of $b_v$ copies of $v$ for all $v\in V(D)$, and join every $L$-side copy of $v$ and $R$-side copy of $w$ whenever $\ori{vw}$ is an arc of $D$.
\end{definition}

We note that each auxiliary bipartite graph $D_{\mathbf{S}}$ counts the spanning subdigraph of $D$ whose out- and in-degrees of $v$ are $a_v$ and $b_v$, respectively, with multiplicity $\prod_{v\in V(D)} a_v! b_v!$. Also, each such subdigraph is counted in at least one perfect matching of $D_{\mathbf{S}}$; we have the following proposition. We denote by $\operatorname{pm}(H)$ the number of perfect matchings of a given graph $H$.

\begin{proposition}\label{prop:aux-perf-matching}
    Let $D$ be an oriented graph and let $\mathbf{S} = \{(a_v, b_v)\}_{v\in V(D)}$ be a collection of pairs of non-negative integers $a_v$ and $b_v$ such that $\sum_{v\in V(D)} a_v = \sum_{v\in V(D)} b_v$. Then the number of subdigraphs of $D$ whose out- and in-degrees of $v$ are $a_v$ and $b_v$, respectively is at most
    $$
        \frac{\operatorname{pm}(D_{\mathbf{S}})}{\prod_{v\in V(D)} a_v! b_v!}.
    $$
\end{proposition}

We now prove Lemma~\ref{lem:edge-contain-prob}

\begin{proof}[Proof of Lemma~\ref{lem:edge-contain-prob}]
    We consider one more parameter $\ve'$ as
    $$
        0 < \frac{1}{r} \ll \frac{1}{k} \ll \ve' \ll \ve < 1.
    $$
    If $S$ contains a vertex whose out- or in-degree exceeds $k$, then the probability is obviously $0$, so we may assume that both the maximum out- and in-degree of $S$ are at most $k$.

    Let $s \defeq |S|$ and set $a_v = k - d^+_S(v)$ and $b_v = k - d^-_S(v)$ for each $v\in V(D)$. Write $\mathbf{S} \defeq \{(a_v, b_v)\}_{v\in V(D)}$. Observe that the auxiliary bipartite graph $D_{\mathbf{S}}$ has maximum degree at most $kr$ and the sequence $(n!)^{1/n}$ is nondecreasing, so Br\'{e}gman--Minc inequality and Proposition~\ref{prop:aux-perf-matching} imply
    $$
        |\{F\in \calF_k(D): E(S)\subseteq E(F)\}| \leq \frac{\left((kr)!\right)^{(kn-s)/(kr)}}{\prod_{v\in V(D)} a_v! b_v!}.
    $$
    As Stirling's approximation gives the inequality $\log(n!) \leq n \log n - n + O(\log(n+1))$, we have
    \allowdisplaybreaks
    \begin{align}
        |\{F\in \calF_k(D): E(S)\subseteq E(F)\}| &\leq \frac{1}{\prod_{v\in V(D)} a_v! b_v!} \cdot \left(\frac{kr}{e} \right)^{kn-s} \cdot \exp\left[O\left(\frac{\log(kr+1) n}{r} \right)\right] \nonumber \\
        &\leq (1 + \ve')^n \cdot \frac{\left((kr/e)^k\right)^{n}}{\prod_{v\in V(D)} a_v! b_v!} \cdot \left( \frac{e}{kr} \right)^{s}. \label{eq:stirling+prop}
    \end{align}
    
    Also, since $a_v! \geq \frac{k!}{k^{d^+_S(v)}}$ and $b_v! \geq \frac{k!}{k^{d^-_S(v)}}$, we have
    $$
        \frac{1}{\prod_{v\in V(D)} a_v!b_v!} \leq \frac{k^{2s}}{(k!)^{2n}}.
    $$
    Thus, together with \eqref{eq:stirling+prop}, we have
    $$
        |\{F\in \calF_k(D): E(S)\subseteq E(F)\}| \leq \left((1 + \ve') A_{r, k}\right)^{n} \left( \frac{ek}{r} \right)^s.
    $$
    Since $|\calF_k(D)| \geq \left((1 - \ve')A_{r, k}\right)^n$ by Lemma~\ref{lem:k-factor-lower}, it holds that
    $$
        \mathbb{P}(E(S)\subseteq E(F)) \leq \left( \frac{1+\ve'}{1 - \ve'}\right)^n \left( \frac{ek}{r} \right)^{s} \leq (1 + \ve)^n \left(\frac{ek}{r}\right)^{|E(S)|}.
    $$
    This completes the proof.
\end{proof}

Lastly, we require the following lemma.

\begin{lemma}\label{lem:imbalanced}
    Fix parameters as
    $$
        0 < \frac{1}{r} \ll \frac{1}{k} \ll \ve < 1.
    $$
    Let $D'$ be an $n$-vertex oriented graph with
    $$
        d_{D'}^+(v)=r+t_v,\qquad d_{D'}^-(v)=r-t_v, \qquad t_v\in\{-1,0,1\},\quad \sum_{v\in V(D')} t_v=0.
    $$ for each $v\in V(D')$.
    Then the number of spanning subdigraphs $J\subseteq D'$ with
    $$
        d_J^+(v)=k+t_v,\qquad d_J^-(v)=k-t_v
    $$ for each $v\in V(D')$ is at most
    $$
        \left((1 + \ve)A_{r, k}\right)^n.
    $$
\end{lemma}

\begin{proof}[Proof of Lemma~\ref{lem:imbalanced}]
    Let $M$ be the number of such spanning subdigraphs $J$, and set $a_v=k+t_v$ and $b_v=k-t_v$; both sums equal $kn$. Consider an auxiliary bipartite graph $D'_{\mathbf{S}}$, where $\mathbf{S} = \{(a_v, b_v)\}_{v\in V(D')}$. We note that every left copy of $v$ in $D'_{\mathbf{S}}$ has degree
    $$
        S_v=\sum_{\ori{vw}\in E(D')}(k-t_w),
    $$
    which lies between $(k-1)(r-1)$ and $(k+1)(r+1)$. 
    Then the Br\'{e}gman--Minc inequality, Proposition~\ref{prop:aux-perf-matching}, and Stirling's approximation imply
    \allowdisplaybreaks
    \begin{align}
        \log M &\leq \sum_v\frac{a_v}{S_v}\log(S_v!) -\sum_v\log(a_v!b_v!)\notag\\
        &\leq \sum_v a_v\log S_v-kn-\sum_v\log(a_v!b_v!) + O\left(\frac{n\log(kr+1)}{r}\right). \label{eq:imbalanced-stirling}
    \end{align}
    Also, for every $v$, we have
    \begin{equation}\label{eq:factorial-pair}
        a_v!b_v!\ge(k!)^2
    \end{equation}
    since the sequence $\{n!\}_{n\in \mathbb{Z}_{\geq 0}}$ is log-convex and $\frac{(a_v + b_v)}{2} = k$.

    Now, it remains to estimate the weighted average of the $S_v$. 
    Since $|E(D')| = rn$ and $\sum_{v\in V(D')} t_v=0$, we have
    \allowdisplaybreaks
    \begin{align}
        \sum_v a_vS_v &=\sum_{\ori{vw}\in E(D')}(k+t_v)(k-t_w)\notag\\
        &= k^2 rn + k\sum_{\ori{vw}\in E(D')}(t_v-t_w)-\sum_{\ori{vw}\in E(D')}t_v t_w\notag\\
        &= k^2 rn + 2k\sum_{v\in V(D')} t_v^2 - \sum_{\ori{vw}\in E(D')}t_v t_w\notag\\
        &\leq k^2 rn + 2kn+ rn. \label{eq:weighted-degree-sum}
    \end{align}
    In the third line, we used
    $$
        \sum_{\ori{vw}\in E(D')}t_v=\sum_{v\in V(D')} (r+t_v)t_v=\sum_vt_v^2,
        \qquad
       \sum_{\ori{vw}\in E(D')}t_w=\sum_{w\in V(D')} (r-t_w)t_w=-\sum_wt_w^2.
    $$
    We now apply Jensen's inequality with weights $a_v/(kn)$. Then \eqref{eq:weighted-degree-sum} yields
    \allowdisplaybreaks
    \begin{align*}
        \sum_{v\in V(D')} a_v\log S_v &\leq kn\log\left(\frac{\sum_{v\in V(D')} a_v S_v}{kn}\right)\\
        &\leq kn\log\left(kr+2+\frac{r}{k}\right)\\
        &\leq kn\log(kr)+\frac{n}{k}+\frac{2n}{r},
    \end{align*}
    where the last inequality holds as $\log(1+x)\le x$. By combining this with \eqref{eq:imbalanced-stirling} and~\eqref{eq:factorial-pair}, we have
    \allowdisplaybreaks
    \begin{align*}
        \log M &\leq \sum_v a_v\log S_v-kn-\sum_v\log(a_v!b_v!) + O\left(\frac{n\log(kr+1)}{r}\right) \\
        &\leq kn\log(kr)+\frac{n}{k}+\frac{2n}{r} - kn - n \log((k!)^2) + O\left(\frac{n\log(kr+1)}{r}\right)\\
        &\leq n \log\left( \frac{(kr/e)^k}{(k!)^2} \right) + n\log(1 + \ve)\\
        &\leq n\log\left( (1 + \ve)A_{r, k} \right).
    \end{align*}
    Thus, we finally deduce that $M \leq \left( (1 + \ve)A_{r, k} \right)^n.$ This completes the proof.
\end{proof}


\subsection{Proof of Theorem~\ref{thm:ratio-EO}}\label{subsec:proof-ratio-EO}

We recall that the high-level idea of the proof of Theorem~\ref{thm:ratio-EO} is to use structural information that most $2k$-factors of a given $2r$-regular graph are likely to satisfy, namely, having a small number of short cycles. Rigorously, we define the following.

\begin{definition}
    For a $2k$-regular graph $F$, we write $U_k(F)$ to be the set of vertices that belong to a cycle of length at most $k^2$.
    If $|U_k(F)| \leq n/k^2$, then we call $F$ \emph{$k$-loose}.
\end{definition}

The following lemma says that in a given $r$-regular oriented graph $D$, the number of $k$-factors whose underlying graph is $k$-loose is sufficiently large relative to the total number of $k$-factors of $D$.

\begin{lemma}\label{lem:loose-many}
    Fix parameters as
    $$
        0 < \frac{1}{r} \ll \frac{1}{k} \ll \ve < 1.
    $$
    Let $D$ be an $n$-vertex $r$-regular oriented graph. Then the number of $k$-factors of $D$ whose underlying graph is $k$-loose is at least
    $$
        \left( (1 - \ve) A_{r, k} \right)^{n}.
    $$
\end{lemma}

\begin{proof}[Proof of Lemma~\ref{lem:loose-many}]
    We additionally fix one more parameter $\ve'$ as $1/k \ll \ve' \ll \ve$.
    Let $H$ be an $n$-vertex $2k$-regular graph which is not $k$-loose. Let $W$ be a subgraph that is a maximal vertex-disjoint union of cycles in $H$ whose lengths are at most $k^2$. Assume $W$ contains $t$ disjoint cycles and consequently $|V(W)| \leq k^2 t$. We note that by maximality of $W$, every cycle of length at most $k^2$ meets $V(W)$, so every vertex in $U_k(H)$ is within distance $k^2$ of $W$. Since $H$ is $2k$-regular, we have
    $$
        \frac{n}{k^2} \leq |U_k(H)| \leq |W|\sum_{i=0}^{k^2} (2k)^i \leq k^2 (2k)^{k^2+1}t.
    $$
    Let $c_k \defeq 1/ (k^4 (2k)^{k^2 + 1})$. Then by the above inequality, $t \geq c_k n$.
    This implies that every $n$-vertex $2k$-regular graph which is not $k$-loose contains at least $\lceil c_k n \rceil$ vertex-disjoint cycles of length at most $k^2$.
    
    Now choose $F$ uniformly from $\calF_k(D)$. Let $\overline{D}$ be the underlying graph of $D$, which is $2r$-regular. We note that the number of cycles of length $\ell$ in $\overline{D}$ is at most $n(2r)^{\ell-1}$. For each such cycle, its edges have uniquely determined orientations in $D$. For a given positive integer $t$ and a vertex-disjoint collection of $t$ cycles of $\overline{D}$, apply Lemma~\ref{lem:edge-contain-prob} to the union of the corresponding directed edge sets. Summing over unordered collections gives
    \allowdisplaybreaks
    \begin{align}
        &\mathbb{P}\bigl(\text{the underlying graph of } F \text{ contains }t \text{ vertex-disjoint cycles of length at most } k^2 \bigr) \notag \\
        &\quad \leq \frac{(1 + \ve')^n}{t!} \left[ \sum_{\ell=3}^{k^2}n(2r)^{\ell-1} \left(\frac{e k}{r}\right)^{\ell} \right]^t \notag \\
        &\quad \leq \frac{(1 + \ve')^n}{t!} \left(\frac{K_kn}{r}\right)^t, \label{eq:cycle-union-bound}
    \end{align}
    where $K_k \defeq \sum_{\ell=3}^{k^2} 2^{\ell-1}(e k)^{\ell}$ depends only on $k$. The factor $1/t!$ is obtained by first summing over ordered collections. Also, dropping the disjointness restriction in this weighted sum only increases it.
    We now take $t$ as $\lceil c_k n \rceil$. Since $t!\ge(t/e)^t$ for every positive integer $t$, \eqref{eq:cycle-union-bound} implies
    \begin{equation}\label{eq:bad-probability}
        \mathbb{P}\bigl(\text{the underlying graph of } F\text{ is not $k$-loose}\bigr) \leq (1 + \ve')^n \left(\frac{e K_k}{ c_k r}\right)^{c_k n} \leq \frac{1}{2}.
    \end{equation}
    By Lemma~\ref{lem:k-factor-lower}, we have $|\calF_k(D)| \geq \left((1 - \ve')A_{r, k}\right)^n$. With this, \eqref{eq:bad-probability} implies that the number of $k$-factors of $D$ whose underlying graph is $k$-loose is at least
    $$
        \frac{1}{2}\cdot \left((1 - \ve')A_{r, k}\right)^n \geq \left((1 - \ve)A_{r, k}\right)^n.
    $$
    This completes the proof.
\end{proof}

The next lemma shows that the notion of $k$-loose is compatible with our desire. The following lemma says that if a graph $H$ is $n$-vertex $2k$-regular $k$-loose and $F$ is an edge-disjoint $n$-vertex $2$-regular graph, then the ratio of the Eulerian orientations in $H$ and $H\cup F$ is at least $(2 - o_k(1))^n$.

\begin{lemma}\label{lem:k-loose-extension}
    Fix parameters as
    $$
        0 < \frac{1}{k} \ll \ve < 1.
    $$
    Let $H$ be an $n$-vertex $2k$-regular $k$-loose graph and $F$ be an $n$-vertex $2$-regular graph on the same vertex set with $H$. Assume $H$ and $F$ are edge-disjoint. Then we have
    $$
        \frac{|\calE(H\cup F)|}{|\calE(H)|} \geq (2 - \ve)^n.
    $$
\end{lemma}

We prove Lemma~\ref{lem:k-loose-extension} by showing that $|\calE(H)| \leq \left((1 + o(1)) 2^{-k} \binom{2k}{k} \right)^n$. Then the conclusion is immediately followed from the lower bound of Theorem~\ref{thm:Schrijver-EO}. 

The notion of $k$-looseness reflects the intuition that graphs with few short cycles should have Eulerian orientation counts close to the lower bound in Theorem~\ref{thm:Schrijver-EO}. The pairing argument below makes this intuition precise: when short cycles are scarce, pairing the incident edges at each vertex decomposes the graph into relatively few closed trails, thereby limiting the additional choices of orientation. 
Meanwhile, the general upper bound in Theorem~\ref{thm:Schrijver-EO} follows from the Br\'{e}gman--Minc inequality applied to an auxiliary bipartite graph obtained by duplicating rows of the vertex--edge incidence matrix. 
Equality in this permanent inequality is attained by disjoint unions of balanced complete bipartite graphs, which contain many short cycles. However, the resulting bound does not exploit the scarcity of short cycles in the original graph. The $k$-loose condition supplies this additional structural information, allowing us to obtain an upper bound close to Schrijver's lower bound. This type of argument, namely preventing short cycles, allows us to approach the lower bound of Theorem~\ref{thm:Schrijver-EO}

The notion of $k$-looseness introduced here relaxes the large-girth condition under which Las Vergnas~\cite{LasVergnas} showed that the lower bound of Theorem~\ref{thm:Schrijver-EO} is asymptotically sharp on the exponential scale, allowing a small exceptional set of vertices on short cycles. We remark that a similar idea was used in~\cite{BBMC,Isaev-McKay-Zhang,HIMZ}.

\begin{proof}[Proof of Lemma~\ref{lem:k-loose-extension}]
    For each vertex $v\in V(H)$, partition its $2k$ incident edges of $H$ into $k$ unordered pairs. We call a choice of such pairings at all vertices a \emph{transition system}. Let $\calP(H)$ be the set of transition systems. As the number of ways to partition $2k$ edges into $k$ unordered pairs is $2^{-k} k! \binom{2k}{k}$, we have
    \begin{equation}\label{eq:num-transition}
        |\calP(H)| = \left( 2^{-k} k! \binom{2k}{k} \right)^n.
    \end{equation}
    
    We note that each transition system $P$ partitions the edge set into closed trails: on entering a vertex along one edge, leave along the edge paired with it. A trail may revisit a vertex, but it uses no edge twice. Let $c(P)$ be the number of closed trails induced by $P$. 
    Observe that each trail has two possible directions, and independently directing all the trails gives an Eulerian orientation compatible with $P$. Conversely, in a fixed Eulerian orientation, there are $k!$ ways at each vertex to pair the incoming edges with the outgoing edges. Thus, it holds that
    \begin{equation}\label{eq:transition-identity}
        |\calE(H)|(k!)^n=\sum_{P\in\calP(H)}2^{c(P)}.
    \end{equation}

    Observe that a closed trail meeting $U_k(H)$ uses at least one of the $k|U_k(H)|$ transition pairs at vertices of $U_k(H)$. Also, different closed trails use different pairs, so at most $k|U_k(H)|\leq n/k$ trails meet $U_k(H)$, since $H$ is $k$-loose. 
    On the other hand, by the definition of the set $U_k(H)$, any closed trail avoiding $U_k(H)$ has length greater than $k^2$. As $H$ has $kn$ edges, at most $n/k$ trails avoid $U_k(H)$. Hence, for every $P\in \calP(H)$, we have $c(P) \leq 2n/k$. Thus, \eqref{eq:num-transition}  and \eqref{eq:transition-identity} imply that
    \begin{equation}\label{eq:upper-EO-H}
        |\calE(H)| = \frac{1}{(k!)^n} \sum_{P\in \calP(H)} 2^{c(P)} \leq \frac{2^{2n/k}}{(k!)^n} |\calP(H)| = 2^{2n/k} \left(2^{-k}\binom{2k}{k} \right)^n.
    \end{equation}

    By Theorem~\ref{thm:Schrijver-EO}, we have
    $$
        |\calE(H\cup F)| \geq \left(2^{-(k+1)} \binom{2(k+1)}{k+1} \right)^n.
    $$
    Together with \eqref{eq:upper-EO-H}, we have
    $$
        \frac{|\calE(H\cup F)|}{|\calE(H)|} \geq 2^{-2n/k} \left(2 - \frac{1}{k+1}\right)^n \geq (2 - \ve)^n.
    $$
    This completes the proof.
\end{proof}

We now prove Theorem~\ref{thm:ratio-EO}.

\begin{proof}[Proof of Theorem~\ref{thm:ratio-EO}]
    Fix parameters as
    $$
        0 < \frac{1}{r} \ll \frac{1}{k} \ll \ve' \ll \ve < 1.
    $$
    Let $G$ be an $n$-vertex $2(r+1)$-regular graph and $G'$ be an $n$-vertex $2r$-regular graph which is a spanning subgraph of $G$. Let $F \defeq G \setminus{G'}$, a $2$-factor of $G$ which is edge-disjoint with $G'$ and $G = G'\cup F$. Write $\calH$ to be the set of all $k$-loose spanning subgraphs of $G'$. We note that all members of $\calH$ are $2k$-factors of $G'$ by the definition of $k$-loose. The key idea of the proof is to consider the two sums
    \begin{align*}
        X &\defeq \sum_{H\in \calH} |\calE(G'\setminus H)|\cdot |\calE(H)|, \\
        Y &\defeq \sum_{H\in \calH} |\calE(G'\setminus H)|\cdot |\calE(H \cup F)|
    \end{align*}
    and estimate the values of $X$ and $Y$.

    Observe that an object counted by $X$ consists of $H\in\calH$, an Eulerian orientation of $H$, and an Eulerian orientation of $G'\setminus H$. Their union is an Eulerian orientation of $G'$, say $D'$. Conversely, for a fixed $D'$, the admissible choices of $H$ correspond exactly to directed $k$-factors of $D'$ with underlying graphs that are $k$-loose. Thus, by double counting and Lemma~\ref{lem:loose-many}, we have
    \begin{equation}\label{eq:X-lower}
        |X| \geq |\calE(G')|\cdot \left((1 - \ve')A_{r, k}\right)^n
    \end{equation}
    On the other hand, since all members of $\calH$ are $k$-loose, by Lemma~\ref{lem:k-loose-extension}, we have $|\calE(H \cup F)| \geq (2 - \ve')^n |\calE(H)|$ for each $H\in \calH$. Then together with \eqref{eq:X-lower}, it holds that
    \begin{equation}\label{eq:Y-lower}
        |Y| \geq (2 - \ve')^n |X| \geq |\calE(G')| \cdot \left((2 - \ve')(1 - \ve') A_{r, k}\right)^n .
    \end{equation}

    We now estimate the upper bound of $Y$. To this end, fix an Eulerian orientation $D$ of $G$. Denote by $\ori{F}$ the subdigraph of $D$ whose underlying graph is $F$, and let $\ori{G'}$ be the oriented graph $D\setminus \ori{F}$. Since $D$ is a fixed oriented graph, the oriented graphs $\ori{F}$ and $\ori{G'}$ are uniquely determined. 
    For each $v\in V(G)$, set $t_v \defeq 1 - d^+_{\ori{F}}(v)$. As $F$ is $2$-regular, $t_v \in \{-1, 0, 1\}$ for all $v\in V(D)$. Also, since $|E(F)| = n$, we have $\sum_{v\in V(G)} t_v = n - |E(F)| = 0$.
    Since $D$ is an $(r+1)$-regular oriented graph, for each vertex $v\in V(\ori{G'})$, the out- and in-degree of $v$ is $r + t_v$ and $r - t_v$, respectively.
    Ignoring the $k$-looseness only increases the number of choices; the number of Eulerian orientations that are contained in $D$ and whose underlying graph is $H\cup F$ for some $H\in \calH$ is at most the number of spanning subdigraphs $J\subseteq \ori{G'}$ with
    $$
        d_J^+(v)=k+t_v,\qquad d_J^-(v)=k-t_v
    $$ for each $v\in V(D)$.
    By Lemma~\ref{lem:imbalanced}, the number of such subdigraphs is at most $\left((1 + \ve')A_{r, k} \right)^n$ 
    
    Similarly to the lower bound deduction of $X$, we use double counting.
    Once Eulerian orientations of $G' \setminus H$ and $H\cup F$ are fixed, their union uniquely determines a fixed Eulerian orientation $D$ of $G$. On the other hand, once we know this fixed $D$ and the Eulerian orientation of $H\cup F$ for a certain $k$-loose subgraph $H$, then the Eulerian orientation of $G' \setminus H$ is determined.
    Once an Eulerian orientation $D$ of $G$ was fixed, it holds that
    \begin{equation}\label{eq:Y-upper}
        |Y| \leq |\calE(G)|\cdot \left((1 + \ve') A_{r, k} \right)^n.
    \end{equation}
    By combining \eqref{eq:Y-lower} and \eqref{eq:Y-upper}, we have
    $$
        \frac{|\calE(G)|}{|\calE(G')|} \geq \left( \frac{(2 - \ve')(1 - \ve')}{(1 + \ve')} \right)^n \geq (2 - \ve)^n.
    $$
    This completes the proof.
\end{proof}


\section{Proof of Theorem~\ref{thm:main}}\label{sec:proof-main}

From now on, we refer to a spanning subdigraph $D'$ of a given digraph $D$ as a $1$-factor of $D$ if all vertices of $D'$ have both out- and in-degree $1$.
To prove the lower bound of Theorem~\ref{thm:main}, we need the following lemma. For a given (di)graph $G$, we denote by $\mathrm{Comp}(G)$ the set of connected components of $G$.

\begin{lemma}\label{lem:comp-exp-upper}
    Let $G$ be an $n$-vertex $2r$-regular graph. Then for all $0  < \gamma \leq 1$, we have
    \begin{equation*}
        \sum_{F\in \calF_2(G)} 2^{|\mathrm{Comp}(F)|} \leq 2^{\gamma n}|\calF_2(G)| + 4^n r^{(1-\gamma)n}.
    \end{equation*}
\end{lemma}

The proof of Lemma~\ref{lem:comp-exp-upper} is similar to the proof of Lemma 4.5 in the author's recent paper~\cite{Lee-KL}.

\begin{proof}[Proof of Lemma~\ref{lem:comp-exp-upper}]
    We partition $\calF_2(G)$ into $\calL \cup \calS$ such that $\calL$ is a collection of $2$-factors in $G$ with at most $\gamma n$ components and $\calS$ consists of the remaining $2$-factors. Obviously, we have
    \begin{equation}\label{eq:calL}
        \sum_{F \in \calL} 2^{|\mathrm{Comp}(F)|} \leq 2^{\gamma n} |\calL| \leq 2^{\gamma n} |\calF_2(G)|.
    \end{equation}
 
    Let $\dir{G}$ be the digraph obtained from $G$ by replacing every edge with a digon. Let $\calS'$ be the set of $1$-factors of $\dir{G}$ whose underlying graph is in $\calS$. Since each cycle has exactly two orientations in which every vertex has indegree and outdegree $1$, it holds that
    \begin{equation}\label{eq:calS'}
        \sum_{F \in \calS} 2^{|\mathrm{Comp}(F)|} = |\calS'|.
    \end{equation}

     Let $R(F')$ be a set of vertices of size exactly $\lceil \gamma n \rceil$ such that each vertex comes from a distinct directed cycle of $F'$. We fix $R(F')$ arbitrarily among possible choices. We denote by $\calS'_R$ the set of $F'\in \calS'$ such that $R(F') = R$. To bound $ |\calS'_R| $, we assume that all vertices outside $R$ choose a vertex from their out-neighborhood arbitrarily. We note that the resulting directed graph after this process can be extended to a member of $\calS'_R$ only if the digraph consists of directed cycles and directed paths, where such directed paths end at the vertices of $R$. Assume that the result digraph satisfies this necessary condition. Since each member of $\calS'_R$ consists of consistently oriented directed cycles, by definition of $\calS'_R$, there is at most one way to extend it to a member of $\calS'_R$: close directed paths into cycles corresponding to each vertex of $R$, if possible. As $\dir{G}$ has maximum out-degree $2r$, it holds that 
     \begin{equation}\label{eq:calS'R}
         |\calS'_R| \leq (2r)^{n-|R|}.
     \end{equation}

     By combining \eqref{eq:calL}--\eqref{eq:calS'R}, we obtain
     \allowdisplaybreaks
     \begin{align*}
         \sum_{F\in \calF_2(G)} 2^{|\mathrm{Comp}(F)|} &\leq 2^{\gamma n} |\calF_2(G)| + |\calS'|\\
         &= 2^{\gamma n} |\calF_2(G)| + \sum_{R\in \binom{V(G)}{\lceil \gamma n \rceil}} |\calS'_R|\\
         &\leq 2^{\gamma n}|\calF_2(G)| + \binom{n}{\lceil \gamma n \rceil} (2r)^{(1 - \gamma)n}\\
         &\leq 2^{\gamma n}|\calF_2(G)| + 4^n r^{(1-\gamma)n}.
     \end{align*}
     This completes the proof.
\end{proof}

We now deduce a lower bound of $\sum_{F\in \calF_2(G)} 2^{|\mathrm{Comp}(F)|}$. Theorem~\ref{thm:ratio-EO} plays a crucial role in the proof.

\begin{lemma}\label{lem:comp-exp-lower}
    Fix parameters as 
    $$
        0 < \frac{1}{r} \ll \ve < 1.
    $$
    Let $G$ be an $n$-vertex $2r$-regular graph. Then we have
    \begin{equation*}
        \sum_{F\in \calF_2(G)} 2^{|\mathrm{Comp}(F)|} \geq \left((1 - \ve) \frac{2r}{e} \right)^n.
    \end{equation*}
\end{lemma}

\begin{proof}[Proof of Lemma~\ref{lem:comp-exp-lower}]
    Let $\calC$ be a set of digraphs which are a disjoint union of consistently oriented directed cycles whose underlying graph is in $\calF_2(G)$. Since each cycle has exactly two distinct consistent orientations, we have that 
    \begin{equation}\label{eq:equality-c}
        |\calC| = \sum_{F\in \calF_2(G)} 2^{|\mathrm{Comp}(F)|}.
    \end{equation}
    
    Let $D \in \calE(G)$ be an oriented graph. Since $D$ is an Eulerian orientation of a $2r$-regular graph, it is an $r$-regular oriented graph. Thus, by the Egorychev--Falikman theorem, we have $|\calF_1(D)| \geq (r/e)^n$, which implies
    \begin{equation}\label{eq:F1D-lower}
        \sum_{D\in \calE(G)} |\calF_1(D)| \geq |\calE(G)| \cdot \left( \frac{r}{e} \right)^n.
    \end{equation}
    On the other hand, by double counting, it holds that
    \begin{equation*}
        \sum_{D\in \calE(G)} |\calF_1(D)| = \sum_{C\in \calC} |\calE(G\setminus \overline{C})|,
    \end{equation*}
    where $\overline{C}$ is an underlying graph of $C\in \calC$.
    Combining \eqref{eq:F1D-lower} gives
    \begin{equation}\label{eq:ratio-mid}
        \left( \frac{r}{e} \right)^n \leq \sum_{C\in \calC} \frac{|\calE(G\setminus \overline{C})|}{|\calE(G)|}.
    \end{equation}
    Since $\overline{C}$ is a $2$-factor of $G$, the graph $G\setminus \overline{C}$ is $2(r-1)$-factor of $G$. Thus, Theorem~\ref{thm:ratio-EO} yields
    \begin{equation*}
        \left( \frac{r}{e} \right)^n \leq \sum_{C\in \calC} \frac{1}{(2(1 - \ve))^n} = \frac{|\calC|}{(2(1 - \ve))^n}.
    \end{equation*}
    This implies the desired inequality $|\calC| \geq \left((1 - \ve)\frac{2r}{e} \right)^n$.
    This completes the proof.
\end{proof}

We now prove Theorem~\ref{thm:main}.

\begin{proof}[Proof of Theorem~\ref{thm:main}]
    Fix parameters as
    $$
        0 < \frac{1}{r} \ll \gamma \ll \ve' \ll \ve < 1.
    $$
    Let $G$ be an $n$-vertex $2r$-regular graph. The upper bound was already proven in Proposition~\ref{prop:BM-upper}. Thus, we prove only the lower bound here.
    
    By combining Lemmas~\ref{lem:comp-exp-upper} and~\ref{lem:comp-exp-lower}, we have
    $$
        \left((1 - \ve')\frac{2r}{e} \right)^n \leq 2^{\gamma n} |\calF_2(G)| + 4^n r^{(1 - \gamma)n}.
    $$
    Since $1/r$ is much smaller than both $\ve'$ and $\gamma$, we obtain
    $$
        \left((1 - 2\ve')\frac{2r}{e} \right)^n \leq 2^{\gamma n} |\calF_2(G)|.
    $$
    Thus, we have
    $$
        |\calF_2(G)| \geq \left(2^{-\gamma}(1-2\ve') \frac{2r}{e} \right)^n \geq \left( (1 - \ve)\frac{2r}{e} \right)^n.
    $$
    This completes the proof.
\end{proof}


\section{Concluding remarks}\label{sec:concluding}

In this paper, we obtain an asymptotically sharp counting version of Petersen's $2$-factor theorem for simple graphs. A natural question is what the exact bounds are. For the upper bound, we conjecture that disjoint unions of complete bipartite graphs maximize the number of distinct $2$-factors among all $2r$-regular simple graphs.

\begin{conjecture}\label{conj:exact}
    Let $G$ be an $n$-vertex $2r$-regular simple graph. Then it holds that
    \begin{equation*}
        |\calF_2(G)|^{\frac{1}{n}} \leq |\calF_2(K_{2r, 2r})|^{\frac{1}{4r}}.
    \end{equation*}
\end{conjecture}
One reason for believing Conjecture~\ref{conj:exact} to be true is that, despite the overcounting due to repeated edges by applying the Br\'{e}gman--Minc inequality twice, the corresponding perfect matching bound in the Br\'{e}gman--Minc inequality is attained by disjoint unions of copies of $K_{2r, 2r}$, and it is unique.
We also believe the $2$-factorization analogue of Conjecture~\ref{conj:exact} is true.

Another natural further direction is considering multigraph cases. If we allow multiple edges, then the lower bound and upper bound can be exponentially different. Denote $J_{n, r}$ as an $n$-vertex $2r$-regular multigraph such that its simplification is a disjoint union of $n/2$ edges and each edge has multiplicity $2r$. Then the number of $2$-factors in $J_{n, r}$ is asymptotically $(\sqrt{2}r)^n$, which is exponentially bigger than $(2r/e)^n$ that appears in the simple graph cases.
Our proof of Theorem~\ref{thm:main} heavily relies on the permanent inequalities that we demonstrated in Section~\ref{sec:permanents}, but the Br\'{e}gman--Minc inequality only works for simple graphs. Because of this issue, the same lower bound cannot be deduced from the same argument as well, but we still believe the lower bound in Theorem~\ref{thm:main} is also valid for the multigraph cases. Rigorously, we pose the following conjecture.

\begin{conjecture}\label{conj:multigraphs}
    Let $M$ be an $n$-vertex $2r$-regular multigraph. Then it holds that
    \begin{equation*}
        \left((1 - o_r(1))\frac{2r}{e} \right)^n \leq |\calF_2(M)| \leq \left((1 + o_r(1))\sqrt{2} r \right)^n. 
    \end{equation*}
\end{conjecture}

\subsection*{Acknowledgement}
A part of this work was done during the author's visit to the University of Vienna. He gratefully acknowledges the University of Vienna for its hospitality and support. 

The author thanks Sang-il Oum for helpful comments on the references.

\subsection*{Statement of AI use}
This work was initiated in January 2026. ChatGPT Pro models were exclusively used in developing the proof of Theorem~\ref{thm:ratio-EO} under the author's multiple guidelines. In particular, the author independently developed all high-level ideas, including the proof of Theorem~\ref{thm:ratio-EO}. 
The author has verified all results and takes full responsibility for the content of the paper.


\vspace{-0.2cm}

\providecommand{\MR}[1]{}
\providecommand{\MRhref}[2]{%
  \href{http://www.ams.org/mathscinet-getitem?mr=#1}{#2}
}

    \bibliographystyle{amsplain_initials_nobysame}
    \addcontentsline{toc}{section}{References}
    \bibliography{bibfile}


\end{document}